\documentclass[11pt]{amsart}

\newtheorem{theorem}{Theorem}[section]

\newtheorem{corollary}[theorem]{Corollary}

\newtheorem{proposition}[theorem]{Proposition}

\theoremstyle{definition}

\usepackage[colorlinks, bookmarks=true]{hyperref}
\usepackage{color,graphicx,shortvrb}

\usepackage{enumerate}
\usepackage{amsmath}
\usepackage{amssymb}

\newcommand{\vertiii}[1]{{\left\vert\kern-0.25ex\left\vert\kern-0.25ex\left\vert #1
    \right\vert\kern-0.25ex\right\vert\kern-0.25ex\right\vert}}

\usepackage[latin 1]{inputenc}

\def\11{\textbf{$1$}}

\begin{document}

\numberwithin{equation}{section}

\title[The $\lambda$-function in the space of trace class operators]{The $\lambda$-function in the space of trace class operators}

\author[A.M. Peralta]{Antonio M. Peralta}
\address{Departamento de An{\'a}lisis Matem{\'a}tico, Universidad de Granada,\\
Facultad de Ciencias 18071, Granada, Spain}
\email{aperalta@ugr.es}

\begin{abstract} Let $C_1(H)$ denote the space of all trace class operators on an arbitrary complex Hilbert space $H$. We prove that $C_1(H)$ satisfies the $\lambda$-property, and we determine the form of the $\lambda$-function of Aron and Lohman on the closed unit ball of $C_1(H)$ by showing that $$\lambda (a) = \frac{1 - \|a\|_1 + 2 \|a\|_{\infty}}{2},$$ for every $a$ in ${C_1(H)}$ with $\|a\|_1 \leq 1$. This is a non-commutative extension of the formula established by Aron and Lohman for $\ell_1$.
\end{abstract}

\date{}
\maketitle

\section{Introduction}\label{sec:intro}

The \emph{$\lambda$-function} was originally introduced by R.M. Aron and R.H. Lohman in \cite{AronLohm}. This function, defined on the unit ball, $\mathcal{B}_{X}$, of a Banach space $X$, is determined by the set, $\partial_e (\mathcal{B}_{X}),$ of all extreme points of $\mathcal{B}_{X}$. For this reason, we restrict ourself to a Banach space $X$ with $\partial_e (\mathcal{B}_{X})\neq \emptyset$. Accordingly to the original definition, a triplet $(e, y,t)\in \partial_e (\mathcal{B}_{X})\times \mathcal{B}_{X} \times [0,1]$ is said to be \emph{amenable} to an element $x$ in $\mathcal{B}_{X}$ if $x= t e + (1-t) y$. The $\lambda$-function is the mapping $\lambda : \mathcal{B}_{X} \to \mathbb{R}_0^{+}$ given by $$\lambda (x) := \sup \Big\{t : (e, y, t) \in \partial_e (\mathcal{B}_{X})\times \mathcal{B}_{X} \times [0,1] \hbox{ is a triplet amenable to } x\Big\}.$$ The space $X$ satisfies the \emph{$\lambda$-property} if $\lambda (x) >0$, for every $x\in \mathcal{B}_{X}$. If $X$ satisfies the $\lambda$-property and $\inf\{\lambda(x): x \in \mathcal{B}_{X} \} > 0$ we say that $X$ has \emph{the uniform $\lambda$-property}.\smallskip

Under additional assumptions on the Banach space $X$, Aron and Lohman shown that for a compact metric space $K$, the spaces $C(K,X)$, $\ell_{\infty}(X),$ and $\ell_{1}(X)$ have the $\lambda$-property. They exhibit spaces failing the $\lambda$-property, for example, $C(\mathcal{B}_{\mathbb{R}^n},\mathbb{R}^n)$ and $C([0,1],\mathbb{R})$ (cf. Remarks 1.7 and 1.10 in \cite{AronLohm}). It is also shown that, for a strictly convex normed space $X$, $\ell_{1}(X)$ has the $\lambda$-property but not the uniform $\lambda$-property.\smallskip

In their initial program, Aron and Lohman also computed the explicit form of the $\lambda$-function in some concrete cases. For example, if $K$ is a compact metric space and $X$ is an infinite-dimensional strictly convex normed space, then for each element $x$ in the closed unit ball of $C(K,X)$ we have \begin{equation}\label{eq lambda function ellinfty} \lambda (x) = \frac{1+m(x)}{2},\hbox{ where } m(x) = \inf\{ \|x(t)\| : t\in K \}
 \end{equation}(see \cite[Theorem 1.6]{AronLohm}), and a similar formula holds for the elements in the closed unit ball of the space $C([0,1],X)$, where $X$ is a strictly convex normed space satisfying dim$(X_{\mathbb{R}}) \geq 2$, and for the space $\ell_{\infty} (Y)$, where $Y$ is a strictly convex normed space (cf. Theorems 1.9 and 1.13  in \cite{AronLohm}, respectively). Another interesting formula is established for the space $\ell_{1} (Y)$, where $Y$ is a strictly convex normed space. In this case, for each $x=(x_n)$ in the closed unit ball of $\ell_{1} (Y)$ we have \begin{equation}\label{eq lambda function in ell1} \lambda (x) = \frac{1-\|x\|_1+ 2 M(x)}{2},
\end{equation} where $M(x) = \sup\{ \|x_n\| : n\in \mathbb{N} \}$ (see \cite[Theorem 1.11]{AronLohm}). By a little abuse of notation and regarding $\ell_{1} (Y)$ inside $\ell_{\infty} (Y)$, we can denote $M(x) = \|x\|_{_{\infty}}$.\smallskip

There is no doubt on the influential role played by the $\lambda$-function in subsequent papers. For the particular purposes of this note, we highlight Question 4.1 in \cite{AronLohm}, where Aron and Lohman ask ``What spaces of operators have the $\lambda$-property and what does the $\lambda$-function look like for these spaces?'' Motivated by this question, L.G. Brown and G.K. Pedersen determined the exact form of the $\lambda$-function on von Neumann algebras and on unital C$^*$-algebras in \cite{Ped91}, \cite{BrowPed95} and \cite{BrowPed97}. Let $A$ denote a unital C$^*$-algebra whose subgroup of invertible elements is denoted by $A^{-1}$. The $\lambda$-function on $\mathcal{B}_A$ can be only understood thanks to the set $A_q^{-1} = A^{-1} \partial_e(\mathcal{B}_{A}) A^{-1}$ of \emph{quasi-invertible elements} in $A$ (see \cite[Theorem 1.1]{BrowPed95}), and a non-commutative extension of the quantifier $m(x)$ employed in \eqref{eq lambda function ellinfty}. For each element $a\in A$, we set $|a| = (a^* a)^{\frac12}$ and \begin{equation}\label{eq def m_q} m_q (a) = \sup\{ \varepsilon >0 :\  ]0,\varepsilon[ \cap \sigma(|a|) = \emptyset \},
\end{equation} where $\sigma(|a|)$ stands for the spectrum of $|a|.$ It was shown by Brown and Pedersen in \cite[Proposition 1.5]{BrowPed95} that $m_q (a) = \hbox{dist} (a, A\backslash A_q^{-1})$ for every $a\in A$. If we write $\alpha_q (a) = \hbox{dist} (a,A_q^{-1})$, the explicit form of the $\lambda$-function on $\mathcal{B}_{A}$ is given by the following expression \begin{equation}\label{eq lambda function Cstaralgebras} \lambda (a) = \left\{\begin{array}{lc}
                     \frac{1+m_q (a)}{2}, & \hbox{ if } a\in \mathcal{B}_A\cap A_q^{-1}; \\
                     \ & \ \\
                     \frac12 (1-\alpha_q (a)), & \hbox{ if } a\in \mathcal{B}_A\backslash A_q^{-1},
                   \end{array}
 \right.
 \end{equation} (cf. \cite[Theorem 3.7]{BrowPed97}). When $A$ is a von Neumann algebra, Pedersen proved in \cite[Theorem 4.2]{Ped91} that $A$ satisfies the uniform $\lambda$-property, and $\lambda(a) =\frac{1+m_q (a)}{2}$ for every element $a$ in $\mathcal{B}_{A}$. The expression in \eqref{eq lambda function Cstaralgebras} is a non-commutative generalization of the formula given in \eqref{eq lambda function ellinfty}.\smallskip

In the wider setting of JB$^*$-triples, the set of \emph{Brown-Pedersen} \emph{quasi-invertible} elements was introduced and developed by H. Tahlawi, A. Siddiqui, and F. Jamjoom (see \cite{TahSiddJam2013}, and \cite{TahSiddJam2014}). By analogy with the setting of C$^*$-algebras, the set of all Brown-Pedersen quasi-invertible elements in a JB$^*$-triple $E$ is denoted by $E_{q}^{-1}$, and $m_q : E\to \mathbb{R}_0^+$ is defined by $m_q (x)=0$ if $x\notin E_q^{-1}$ and $$m_q (x) = \left(\gamma^{q}(x)\right)^{\frac12}= \min\{t : \ t\in \hbox{Sp}(x) \}= \sup\Big\{\varepsilon>0 : ]0, \varepsilon[\cap \hbox{Sp}(x) = \emptyset\Big\},$$ if $x\in E_q^{-1}$,  where $\gamma^{q}(x)$ is the quadratic conorm introduced and developed in \cite{BurKaMoPeRa} and $\hbox{Sp}(x)$ denotes the triple spectrum of $x$. It is shown in \cite[Theorem 3.1]{JamPerSidTah2015} that $m_q(x) = \hbox{dist} (x, E\backslash E_q^{-1})$ for every $x\in E$. It is proved in \cite[Theorem 4.2]{JamPerSidTah2015} that in a JBW$^*$-triple $W$ the $\lambda$-function is determined by the following expression \begin{equation}\label{eq lambda function JBWtriples} \lambda (a) = \left\{\begin{array}{lc}
                     \frac{1+m_q (a)}{2}, & \hbox{ if } a\in \mathcal{B}_{W}\cap W_q^{-1} \\
                     \ & \ \\
                     \frac12 (1-\alpha_q (a))=\frac12, & \hbox{ if } a\in \mathcal{B}_{W}\backslash W_q^{-1},
                   \end{array}
 \right.
\end{equation}  where $\alpha_q (x) = \hbox{dist} (x,W_q^{-1})$ ($x\in W$). This is a new generalization of the formula given in \eqref{eq lambda function ellinfty}. The form of the $\lambda$-function on the closed unit ball of a general JB$^*$-triple remains as an open problem. Some advances are established for extremally rich JB$^*$-triples (see \cite{JamPerSidTah2016}). For the sake of conciseness, we shall not recall the explicit definition of JB$^*$-triple introduced in \cite{Ka83}, nor the notion of the triple spectrum. The reader interested in such details can consult \cite{BurKaMoPeRa,JamPerSidTah2015}.\smallskip

The question posed by Aron and Lohman on how the $\lambda$-function looks like for other spaces of operators is perfectly valid for the space $C_1(H),$ of all trace class operators on a complex Hilbert space $H$, when this space is regarded as a non-commutative generalization of $\ell_1$. It seems natural to conjecture if the formula given in \eqref{eq lambda function in ell1} is valid for $C_1(H)$. In this note we establish a non-commutative version of this formula by showing that for every complex Hilbert space $H,$ the $\lambda$-function on $\mathcal{B}_{C_1(H)}$ satisfies $$\lambda (a) = \frac{1 - \|a\|_1 + 2 \|a\|_{\infty}}{2},$$ for every $a\in {C_1(H)}$ with $\|a\|_1 \leq 1$ (see Theorem \ref{t lambda function in infinite dimensional case}). This closes a natural conjecture which has been considered in recent years. We provide two different approaches to this result. In the first one we first prove the result in the case in which $H$ is finite-dimensional (see Theorem \ref{t lambda function in finite dimensional case}). In the proof we employ an inequality for the norm in $C_p(\ell_2^n)$ due to L. Mirsky (see Theorem \ref{t Mirsky}), and we build upon it to derive the general case. In the second section we make use of the very basic theory of operator spaces to extend Mirsky's inequality to $C_p(H)$ when $H$ is an  arbitrary complex Hilbert space (see Theorem \ref{t Markus non selfadjoint}). Based on this new inequality we can obtain an alternative proof for our main result.

\section{A lower bound for the $\lambda$-function and a proof through the finite-dimensional case}

Let us fix a complex Hilbert space $H$. The space $C_1(H)$ of all trace class operator on $H$ can be regarded as a non-commutative alter ego of $\ell_1$. This is a particular case of a more general construction which is given to define the space $C_p(H)$ of $p$-Schatten von Neumann operators on $H$ for every $1\leq p<\infty$. As commented in the introduction, the \emph{modulus} of an element $a$ in a C$^*$-algebra $A$ is given by $|a| = (a^*a)^{\frac12}$. If $a$ is an element in the space $K(H),$ of all compact linear operators on $H$, its modulus also is a compact operator, and consequently, it spectrum, $\sigma(|a|),$ is at most countably infinite. The \emph{singular values} of the operator $a$
are precisely the eigenvalues of $|a|$ arranged in decreasing order and repeated according to multiplicity. We write $(\mu_n (a))_n$ for the sequence of singular values of $a$. Since $a\in K(H)$, the sequence $(\mu_n (a))_n$ belongs to $c_0$, and in some cases $(\mu_n (a))_n$ lies in $c_{00}$.\smallskip

For $p\in [1,\infty)$, the set $C_p(H)$ is defined by $$C_p(H) := \left\{ a\in K(H) : \hbox{tr}(|a|^p) = \|a\|_p^p := \left(\sum_{n=1}^{\infty} |\mu_n (a)|^p \right) <\infty \right\},$$ where tr$(.)$ denotes the usual trace on $B(H)$. $C_p(H)$ is a Banach space when equipped with the norm  $\|.\|_p$ (see \cite[Corollary 3.2]{McCarthy67}), and it is called the space of $p$-Schatten von Neumann operators. The predual $B(H)_*$ of $B(H)$ (and the dual $K(H)^*$ of $K(H)$) can be identified with $C_1(H)$ under the isometric linear mapping $a\mapsto \varphi_a$, with $\varphi_a (x) =\hbox{tr} (x a)$ for all $x\in B(H)$, $a\in C_1(H)$ (cf. \cite[Theorem 1.15.3]{S} or \cite[Theorems II.1.6 and II.1.8]{Tak}). Following standard notation, the symbol $C_{\infty} (H)$ will stand for the space $B(H)$, and we write $\|.\|_{\infty}$ for the operator norm on $B(H)$. Let us observe that $C_p(H)\subset K(H)\subset B(H)$ for all $1\leq p<\infty$, and thus $\|a\|_{\infty}$ makes sense for every $a\in C_1(H)$.\smallskip

Elements $a,b\in B(H)$ are called \emph{orthogonal} if $a b^* = b^* a =0$. Let us observe that if $a$ and $b$ are orthogonal elements in $C_p(H)$, then $\|a\pm b\|_p^p = \|a\|_p + \|b\|_p$. For those notions not fully detailed in this paper the reader is referred to \cite{McCarthy67}, \cite[Chapter III]{GohbergKrein}, \cite[\S 9]{DunSchw63}, \cite[Chapter II]{Tak} and \cite[\S 1.15]{S}.\smallskip

C.A. McCarthy proved in \cite{McCarthy67} that, for $1<p<\infty$, the space $C_p(H)$ is uniformly convex and thus it is strictly convex. For a strictly convex normed space $X$ the $\lambda$-function was completely determined by Aron and Lohman, who proved that $\lambda (x) =\frac{1+\|x\|}{2}$ for all $x\in \mathcal{B}_{X}$ (cf. \cite[Proposition 1.2$(e)$]{AronLohm}). Therefore the $\lambda$-function is completely determined on the closed unit ball of $C_p(H)$ for $1<p<\infty$. However, the case $p=1$ is completely different and requires a new geometric argument.\smallskip

Let us recall some notation. Given $\xi,\eta$ in a Hilbert space $H$, the symbol $\eta\otimes \xi$ will denote the element in $C_1(H)= K(H)^* = B(H)_*$ defined by $\eta\otimes \xi (x) = \langle x(\eta) |\xi\rangle$ ($x\in B(H)$). Every trace class operator $a$ in $C_1(H)$ can be written as a (possibly finite) sum
\begin{equation}\label{eq spectral resolution in trace class}
a= \sum_{n=1}^{\infty} \mu_n \eta_n \otimes \xi_n,
\end{equation} where $(\mu_n)\subset \mathbb{R}_0^{+}$,
$(\xi_n)$, $(\eta_n)$ are orthonormal systems in $H$,  and
$\displaystyle \|a\|_1=\sum_{n=1}^{\infty} \mu_n$ (see, for example, \cite[Corollary 1.15.5]{S} or \cite[Theorem
II.1.6]{Tak}).\smallskip

In order to deal with the $\lambda$-function we shall also need a good knowledge of the set of extreme points of the closed unit ball of $C_1(H)$. It is known that $\partial_e (\mathcal{B}_{C_1(H)}) = \left\{ \eta\otimes \xi : \eta, \xi\in H, \ \|\eta\|=\| \xi\|=1   \right\}.$\smallskip

We can establish next a lower bound of the $\lambda$-function on $\mathcal{B}_{C_1(H)}$.

\begin{proposition}\label{p lower bound for lambda} Let $H$ be a complex Hilbert space. Then $C_1(H)$ satisfies the $\lambda$-property, and the inequality $$\lambda (a) \geq \frac{1 - \|a\|_1 + 2 \|a\|_{\infty}}{2},$$ holds for every $a\in \mathcal{B}_{C_1(H)}$.
\end{proposition}

\begin{proof} We begin by showing that $C_1(H)$ satisfies the $\lambda$-property. To this aim, let us fix an arbitrary $a\in \mathcal{B}_{C_1(H)}$, and let us write $a$ in the form $\displaystyle a= \sum_{n=1}^{\infty} \mu_n \eta_n \otimes \xi_n,$ where $(\mu_n)\subset \mathbb{R}_0^{+}$,
$(\xi_n)$, $(\eta_n)$ are orthonormal systems in $H$, and $\displaystyle \|a\|_1 =\sum_{n=1}^{\infty} \mu_n \leq 1$ (let us observe that $\mu_1,\ldots,\mu_n$ do not coincide, in general, with the singular values of $a$ because they can appear in different order). If $\mu_1=1$, then $a= \eta_1 \otimes \xi_1\in \partial_e(\mathcal{B}_{C_1(H)})$, and hence $\lambda (a) = 1$ (see \cite[Proposition 1.2$(a)$]{AronLohm}). There is no loss of generality in assuming that $1> \mu_1>0$. The element $e_1 = \eta_1 \otimes \xi_1$ lies in $\partial_e(\mathcal{B}_{C_1(H)})$, and we can write $a = \mu_1 e_1 + (1-\mu_1) y$ where $\displaystyle y = \sum_{n=2}^{\infty} \frac{\mu_n}{1-\mu_1} \eta_n \otimes \xi_n.$ Since, for every $n\neq m$, $\eta_n \otimes \xi_n$ is orthogonal to $\eta_m \otimes \xi_m$ it can be easily seen that $\displaystyle \|y\|_1 = \sum_{n=2}^{\infty} \frac{\mu_n}{1-\mu_1} = 1.$ Therefore $\lambda (a) \geq \mu_1>0,$ which shows that $C_1(H)$ satisfies the $\lambda$-property.\smallskip

Let $Y$ denote the norm closed subspace of $C_1(H)$ generated by $\{ \eta_n \otimes \xi_n : n\in \mathbb{N}\}$. Clearly $Y$ is isometrically isomorphic to $\ell_1$ and $a\in Y$. Therefore $Y$ satisfies the $\lambda$-property by \cite[Theorem 1.11]{AronLohm}. To avoid confusion, let $\lambda_Y$ and $\lambda_{C_1(H)}$ denote the $\lambda$-functions of $Y$ and $C_1(H)$, respectively. Another clear property is that $\partial_e(\mathcal{B}_Y) \subset \partial_e(\mathcal{B}_{C_1(H)})$. We are in position to apply Proposition 1.2$(f)$ to deduce that $\lambda_Y (x) \leq \lambda_{C_1(H)} (x)$ for every $x\in \mathcal{B}_Y$, and in particular, $\lambda_Y (a) \leq \lambda_{C_1(H)} (a)$. Identifying $Y$ with $\ell_1$, an application of \cite[Theorem 1.11]{AronLohm} implies that $\lambda_Y (a) = \frac{1-\|a\|_1+ 2 M(a)}{2},$ where $M(a) = \sup\{\mu_n : n\in \mathbb{N} \}=\max\{\mu_n : n\in \mathbb{N} \}$.\smallskip

Finally, when $a$ is regarded as an operator in $K(H)$, it can be easily check that $\|a\|_{\infty} = \sup\{\mu_n : n\in \mathbb{N} \}.$ We have therefore shown that $$\lambda_{C_1(H)} (a) \geq \frac{1-\|a\|_1+ 2 \|a\|_{\infty}}{2},$$ which concludes the proof.
\end{proof}

The reciprocal inequality to that established in Proposition 2.1 will be obtained with techniques of spectral theory.\smallskip

Let us observe that every element $e\in \partial_e(\mathcal{B}_{C_1(H)})$, when regarded in $K(H)$, is a minimal (i.e. rank-one) partial isometry, that is, $e e^* e =e$. Actually, every minimal partial isometry $e\in K(H)$ lies in the unit sphere of $C_p(H)$, for every $1\leq p\leq \infty$. Let $\mathcal{U}_{min} (H)$ denote the set of all minimal partial isometries in $B(H)$. In a recent paper, jointly written with F.J. Fern{\'a}ndez-Polo and E. Jord{\'a}, we prove that for a finite-dimensional complex Hilbert space $H$, $1<p<\infty$, and fixed $a\in C_p(H)$ and $\gamma \geq 1$, the minimum value of the function $f_{a,\gamma} : \mathcal{U}_{min} (H)\to \mathbb{R}_0^+,$ $e\mapsto f_{a,\gamma}(e):=\|a-\gamma e\|_p^p,$ is $$ (\gamma-\mu_1(a))^p+\sum_{j=2}^n \mu_j(a)^p = (\gamma-\mu_1(a))^p + (\|a\|_p^p-\mu_1(a)^p),$$ where $\mu_1(a)\geq \ldots \geq \mu_n(a)$ are the singular values of $a$ (cf. \cite[$(11)$ in the proof of Proposition 2.14]{FerJorPer2018}). Thanks to the differentiability of the norm $\|.\|_p$, together with a result of R. Bhatia and P. Semrl in \cite{BhatSemr96}, the points at which $f_{a,\gamma}$ attains its minimum value are completely determined in \cite[$(12)$ in the proof of Proposition 2.14]{FerJorPer2018}. The lacking of a good differentiability of the norm of $C_1(H)$ makes invalid parts of the arguments applied in the case of $C_p(H)$ with $1<p<\infty$. However, the minimum value of this mapping can be also determined in the proof of our next result.\smallskip

We recall first a theorem due to L. Mirsky.

\begin{theorem}{\rm(\label{t Mirsky} or \cite{Mirsky60}\cite[Theorem 9.8]{Bhat2007})} Let $a$ and $b$ be $n\times n$ matrices with complex entries, and let $\vertiii{.}$ be a unitarily-invariant norm on $M_n(\mathbb{C})$. Then $$ \vertiii{\hbox{diag} (\mu_1(a), \ldots, \mu_n(a)) - \hbox{diag} (\mu_1(b), \ldots, \mu_n(b))} \leq \vertiii{a-b},$$ where $\hbox{diag} (., \ldots, .)$ stands for the diagonal matrix whose entries are given by the corresponding list.$\hfill\Box$
\end{theorem}

\begin{theorem}\label{t lambda function in finite dimensional case} Let $H$ be a finite-dimensional complex Hilbert space. Then the identity $$\lambda (a) = \frac{1 - \|a\|_1 + 2 \|a\|_{\infty}}{2},$$ holds for every $a\in \mathcal{B}_{C_1(H)}$.
\end{theorem}

\begin{proof} Proposition \ref{p lower bound for lambda} assures that \begin{equation}\label{eq bound below in thm} \lambda (a) \geq \frac{1 - \|a\|_1 + 2 \|a\|_{\infty}}{2},
 \end{equation} for every $a\in \mathcal{B}_{C_1(H)}$. We shall next prove the reciprocal inequality.\smallskip

We identify $B(H)$ with $M_n (\mathbb{C})$. Let us fix $a\in {C_1(H)}$ and $t\in ]0,1]$. We consider the mapping $f_{a,t} : \mathcal{U}_{min} (H)\to \mathbb{R}_0^+,$ $f_{a,t}(e):=\|a- t e\|_1$. We claim that \begin{equation}\label{eq minimum value of fa p=1}\hbox{ the minimum value of the mapping $f_{a,t}$ in $\mathcal{U}_{min} (H)$ is}
\end{equation} $$\displaystyle |t-\mu_1(a)|+\sum_{j=2}^n \mu_j(a) = |t-\mu_1(a)| + (\|a\|_1-\mu_1(a)),$$ where $\mu_1(a)\geq \ldots \geq \mu_n(a)$ are the singular values of $a$.\smallskip

To see the claim let $e$ be a minimal partial isometry in $B(H)$. We write $a$ in the form $\displaystyle a= \sum_{n=1}^m \mu_n (a) \eta_n\otimes \xi_n$, where $\{\xi_1,\ldots,\xi_m\}$ and $\{\eta_1,\ldots,\eta_m\}$ are orthonormal basis of $H$. We recall that the norm $\|.\|_1$ is unitarily-invariant (see \cite[page 28]{Bhat2007}). We can therefore apply Mirsky's theorem (see Theorem \ref{t Mirsky}) to $a$ and $e$ to conclude that $$ \| \hbox{diag} (\mu_1(a), \ldots, \mu_n(a)) - \hbox{diag} (\mu_1(t e), \ldots, \mu_n(t e)) \|_1 \leq \|a- t e\|_1.$$ Since $\hbox{diag} (\mu_1(t e), \ldots, \mu_n(t e)) = \hbox{diag} (t, 0, \ldots, 0) $, it follows that $$f_{a,t} (\eta_1\otimes \xi_1) =  |t-\mu_1(a)|+\sum_{j=2}^n \mu_j(a) = |t-\mu_1(a)| + (\|a\|_1-\mu_1(a)) $$ $$= \| \hbox{diag} (\mu_1(a), \ldots, \mu_n(a)) - \hbox{diag} (\mu_1(t e), \ldots, \mu_n(t e)) \|_1 \leq \|a-t e\|_1 = f_{a,t} (e),$$ which proves the claim in \eqref{eq minimum value of fa p=1}.\smallskip

Let us observe that $\|a\|_{\infty} = \max\{\mu_1 (a),\ldots,\mu_n(a)\} = \mu_1 (a).$\smallskip

We assume first that $\|a\|_1 <1$. We deduce from \eqref{eq bound below in thm} that $$\lambda (a) \geq \frac{1 - \|a\|_1 + 2 \|a\|_{\infty}}{2} > \|a\|_{\infty} = \mu_1(a).$$ Therefore, in order to compute the supremum in the definition of $\lambda(a)$, we can reduce our study to those triplets $(e, y, t)\in \partial_e (\mathcal{B}_{C_1(H)})\times \mathcal{B}_{C_1(H)} \times ]0,1]$ which are amenable to $a$ with $t> \mu_1(a)$. Under these assumptions $a = t e + (1-t) y $ and hence $f_{a,t} (e) = \|a -t e \|_1 = \|(1-t) y \|_1 \leq 1-t$, which by \eqref{eq minimum value of fa p=1} implies that $$ t-\mu_1(a) + (\|a\|_1-\mu_1(a))  = |t-\mu_1(a)| + (\|a\|_1-\mu_1(a)) \leq f_{a,t} (e) \leq 1-t,$$ or equivalently $$ 2 t \leq 1 - \|a\|_1 + 2\mu_1(a) = 1 - \|a\|_1 + 2 \| a\|_{\infty}.$$ We have shown that $$\lambda (a) = \sup \left\{t : \begin{array}{c}
                                                                                                                    (e, y, t) \in \partial_e (\mathcal{B}_{X})\times \mathcal{B}_{X} \times [0,1] \\
                                                                                                                    \hbox{ amenable to } a
                                                                                                                  \end{array}\right\}\leq  \frac{1 - \|a\|_1 + 2 \|a\|_{\infty}}{2},$$
and thus $\lambda (a) = \frac{1 - \|a\|_1 + 2 \|a\|_{\infty}}{2}$.\smallskip

Let us assume next that $\|a\|_1 =1$. By \eqref{eq bound below in thm} we have $\lambda (a) \geq \|a\|_{\infty} = \mu_1(a)$. If $\lambda (a) > \|a\|_{\infty} = \mu_1(a),$ then there exists a triplet $(e, y, t)\in \partial_e (\mathcal{B}_{C_1(H)})\times \mathcal{B}_{C_1(H)} \times ]0,1]$ which is amenable to $a$ and $t>\mu_1(a)=\|a\|_{\infty}$. In this case, the claim in \eqref{eq minimum value of fa p=1} assures that
$$ t-\mu_1(a) + (\|a\|_1-\mu_1(a))  \leq f_{a,t} (e)= \|a-t e\|_1 = \|(1-t) y\|_1 \leq 1-t,$$ and hence $2t \leq 2 \mu_1(a) = 2 \|a\|_{\infty}$, which is impossible. Consequently, $\lambda (a) = \|a\|_{\infty} = \frac{1 - \|a\|_1 + 2 \|a\|_{\infty}}{2}.$
\end{proof}

It is known that every finite-dimensional Banach space satisfies the uniform $\lambda$-property (cf. \cite[Theorem 1.16]{AronLohm}); this is the case of $C_1(H)$ when $H$ is finite-dimensional. Therefore, the conclusion in the above Theorem \ref{t lambda function in finite dimensional case} can be strengthened by saying that $C_1(H)$ satisfies the uniform $\lambda$-property if $H$ is finite-dimensional. We shall see later that, as in the case of $\ell_1(X)$ in \cite[Theorem 1.11]{AronLohm}, the space $C_1(H)$ fails the uniform $\lambda$-property when $H$ is infinite-dimensional.\smallskip

Helped by the result proved in the finite-dimensional case, we can now establish our main result in full generality.

\begin{theorem}\label{t lambda function in infinite dimensional case} Let $H$ be a complex Hilbert space. Then the identity $$\lambda (a) = \frac{1 - \|a\|_1 + 2 \|a\|_{\infty}}{2},$$ holds for every $a\in \mathcal{B}_{C_1(H)}$. Moreover, $\lambda(a)$ is attained.
\end{theorem}

\begin{proof} As before, Proposition \ref{p lower bound for lambda} assures that \begin{equation}\label{eq bound below in thm 2} \lambda (a) \geq \frac{1 - \|a\|_1 + 2 \|a\|_{\infty}}{2},
 \end{equation} for every $a\in \mathcal{B}_{C_1(H)}$.\smallskip

Let us fix $a\in \mathcal{B}_{C_1(H)}$. As we commented above, there exist orthonormal systems $(\xi_n)$ and $(\eta_n)$ in $H$ such that
$\displaystyle a= \sum_{n=1}^{\infty} \mu_n (a) \eta_n \otimes \xi_n,$ where $(\mu_n (a) )\subset \mathbb{R}_0^{+}$ is the sequence of singular values of $a$, $\displaystyle \|a\|_1 =\sum_{n=1}^{\infty} \mu_n (a) \leq 1$, and $\|a\|_{\infty} = \mu_1(a)$. To simplify the notation we write $e_n = \eta_n \otimes \xi_n$.\smallskip

Fix an arbitrary triplet  $(e, y, t)\in \partial_e (\mathcal{B}_{C_1(H)})\times \mathcal{B}_{C_1(H)} \times ]0,1]$ which is amenable to $a$, that is, $a = t e + (1-t) y$. Given $\varepsilon >0$, let us find a finite rank projection $p\in B(H)$, depending on $\varepsilon$, $e,$ and $a$, such that $\|p a p \|_{\infty} =\|a\|_{\infty} =\mu_1(a)$, $\| a- p a p \|_1 \leq \varepsilon$, and $p e = e p = e.$ Let $K$ denote the finite-dimensional Hilbert space $p(H)$. Since the mapping $C_1(H)\to C_1(H)$, $x\mapsto pxp$ is a contractive linear projection on $C_1(H)$, we deduce that $\|p y p \|_1\leq \|y\|_1\leq 1$. We also know that $$ p a p = t p e p + (1-t) p y p= t e + (1-t) p y p,$$ where $p a p, e, p y p\in C_1(K)$ and $e\in \partial_e(\mathcal{B}_{C_1(K)})$. Therefore, $(e, p y p, t)$ is a triplet in $\partial_e (\mathcal{B}_{C_1(K)})\times \mathcal{B}_{C_1(K)} \times ]0,1]$ which is amenable to $p a p$. Theorem \ref{t lambda function in finite dimensional case} implies that $$ t \leq \lambda_{C_1(K)} (p a p) = \frac{1 - \| p a p\|_1 + 2 \|p a p\|_{\infty}}{2} < \frac{1 - \| a \|_1 +\varepsilon + 2 \| a \|_{\infty}}{2}. $$ We deduce from the arbitrariness of $\varepsilon>0$ that $t\leq \frac{1 - \| a \|_1  + 2 \| a \|_{\infty}}{2},$ for every $t$ in a triplet $(e, y, t)\in \partial_e (\mathcal{B}_{C_1(H)})\times \mathcal{B}_{C_1(H)} \times ]0,1]$ amenable to $a$, which proves that $\lambda(a) = \frac{1 - \| a \|_1  + 2 \| a \|_{\infty}}{2}$.\smallskip

To see the last statement, as in the proof of \cite[Theorem 1.11]{AronLohm}, there is no loss of generality in assuming that $a\notin \partial_e(\mathcal{B}_{C_1(H)})$, and under this assumption, we take $e_1$ and $$ y = \frac{\|a\|_1-1}{ (1+\|a\|_1-2 \|a\|_{\infty})}  e_1 +  \sum_{n\geq 2} \frac{2 \mu_n(a)}{1+\|a\|_1-2 \|a\|_{\infty}} e_n.$$ It is not hard to check that $\|y\|_1= 1$, and $a = \lambda(a) e_1 + (1-\lambda(a)) y$, witnessing that $\lambda(a)$ is attained.
\end{proof}

It seems appropriate to insert a couple of comments. First, the $\lambda$-function is not, in general, a continuous function, nor even in the case of finite-dimensional spaces (see, for example, \cite[Remark 2.4 and Theorem 2.10]{AronLohm}). Anyway, the $\|.\|_1$-continuity of the $\lambda$-function on $C_1(H)$ follows as a consequence of Theorem \ref{t lambda function in infinite dimensional case} above. We observe that in the proof in the above Theorem \ref{t lambda function in infinite dimensional case} the strategy consisted in approximating, in the norm $\|.\|_1$, an arbitrary element $a$ in $\mathcal{B}_{C_1(H)}$ by an element of the form $p a p,$ for a suitable finite rank projection $p\in B(H)$. A subtle approximation allows us to apply the result in the finite-dimensional case established in Theorem \ref{t lambda function in finite dimensional case}. However, in our arguments we have not make any use of the continuity of the $\lambda$-function.\smallskip

Like in the case of $\ell_1(X)$, where $X$ is a strictly convex normed space (cf. \cite[Theorem 1.11]{AronLohm}), when $H$ is infinite-dimensional, the space $C_1(H)$ does not satisfy the uniform $\lambda$-property. To see this, for each natural $n$, we just take an element $a_n$ in the unit sphere of $C_1(H)$ such that $\|a_n\|_{\infty} = \frac1n$ and $\lambda(a) = \frac1n$ (consider, for example, $\displaystyle a_n=\sum_{k=1}^n \frac1n \xi_n\otimes\xi_n$, where $\{\xi_1,\dots,\xi_n\}$ is an orthonormal system in $H$).

\section{A glimpse at operator spaces}

In the arguments developed in the previous section to determine the form of the $\lambda$-function on $\mathcal{B}_{C_1(H)}$ we settled upon a proof which passes through the finite-dimensional case and a subsequent approximation argument. We shall complete this note with an alternative argument relying on basic theory of operator spaces, which might have its own interest for other problems.\smallskip

Let $A$ be a C$^*$-algebra regarded as a closed hermitian subalgebra of some $B(H)$. Let $M_n (A)$ denote the space of all $n\times n$-matrices with entries in $A$. With the obvious matrix multiplication, and the $^*$-operation, $M_n(A)$ is $^*$-algebra (cf. \cite[\S IV.3]{Tak}). For each natural $n$, let $H^{(n)}= \ell_2^n(H)$ be the Hilbert space direct sum of $n$ copies of $H$. The $^*$-algebra $M_n(A)$ is a closed C$^*$-subalgebra of $B(H^{(n)})$ (this corresponds to the so-called \emph{canonical operator space structure} on $A$ in \cite{BleLeMerd2004}). A \emph{concrete operator space} is a closed subspace $X$ of some $B(H)$. The space $M_n(X)$ is regarded as a closed subspace of $M_n(B(H))\cong B(H^{(n)})$ with the norm $\|.\|_n$. We are mainly interested on $M_2 (X)$ with $X=C_1(H)$. For the basic background on operator spaces we refer to \cite{BleLeMerd2004, EffRuan2000} and \cite{Pi2003}.\smallskip

A useful device due to Wielandt asserts that if $a$ is any matrix in $M_n(\mathbb{C})$ then the matrix $\widetilde{a} = \left(
                                                                                          \begin{array}{cc}
                                                                                            0 & a \\
                                                                                            a^* & 0 \\
                                                                                          \end{array}
                                                                                        \right)\in M_2(M_n(\mathbb{C}))$
is hermitian and the eigenvalues of $\widetilde{a}$ are the singular values of $a$ together with their negatives (see, for example, \cite[Proof of Corollary 3.7]{Bhat2007}).\smallskip

Let $a$ be an element in $K(H)$. We consider the space $M_2 (K(H))$ with the C$^*$-structure derived from $M_2 (B(H)) \cong  B(H^{(2)})$. It is easy to see that the element $\widetilde{a} = \left(
                                                                                          \begin{array}{cc}
                                                                                            0 & a \\
                                                                                            a^* & 0 \\
                                                                                          \end{array}
                                                                                        \right)\in M_2(B(H)) \cong   B(H^{(2)})$
                                                                                        is
a hermitian compact operator, and hence lies in $K(H^{(2)})$. Suppose now that $a\in C_p(H)$ with $1\leq p<\infty$. Then we can write $\displaystyle a= \sum_{n=1}^{\infty} \mu_n (a) \eta_n \otimes \xi_n,$ where $(\mu_n (a) )\subset \mathbb{R}_0^{+}$ is the sequence of singular values of $a$, $\{\xi_n\}_n$ and $\{\eta_n\}_n$ are orthonormal systems in $H$, and $\displaystyle \|a\|_p^p =\sum_{n=1}^{\infty} \mu_n(a)^p$ (with $\|a\|_{\infty} = \mu_1(a)$), and consequently, $\displaystyle a^*= \sum_{n=1}^{\infty} \mu_n (a) \xi_n \otimes \eta_n.$ Let us denote the elements in $H^{(2)}$ by pairs of the form $(\xi, \eta)$ with $\xi, \eta\in H$. It can be easily seen that \begin{equation}\label{eq spectral resol tilde a} \widetilde{a} = \sum_{n=1}^{\infty} \mu_n (a) ((\eta_n,0) \otimes (0,\xi_n)+ (\eta_n,0) \otimes (0,\xi_n)),
\end{equation} and hence $\widetilde{a}\in C_p(H^{(2)})$.\smallskip

We recall some notation from \cite[\S 41]{Bhat2007}. Let $a$ be a compact self-adjoint operator in $B(H)$. An \emph{enumeration of the eigenvalues of $a$} is a family of the form $\{ \alpha_n(a) : n\in \pm \mathbb{N}\}$ satisfying the following properties:\begin{enumerate}[$(i)$]\item $\alpha_1 (a) \geq \alpha_2 (a)\geq \ldots\geq 0$, $\alpha_{-1} (a) \leq \alpha_{-2} (a)\leq \ldots\leq 0$;
\item All non-zero eigenvalues of $a$ are included in these sequences with their proper multiplicities;
\item If $a$ has infinitely many positive eigenvalues, then these make up
the entire sequence $\{\alpha_n : n\in \mathbb{N}\}$; but if $a$ has finitely many
positive eigenvalues, then in addition to them this sequence contains infinitely many zeros;
\item Similarly the sequence $\{\alpha_{-n} : n\in \mathbb{N}\}$ contains only the negative
eigenvalues of $a$ if there are infinitely many of them, and if there
are only a finite number of these, then it contains in addition to them
infinitely many zeros.
\end{enumerate}

Having in mind the expression in \eqref{eq spectral resol tilde a} we can easily deduce that $\widetilde{a}$ admits an enumeration of its eigenvalues of the form $\{ \alpha_n(a) : n\in \pm \mathbb{N}\},$ where $\{\alpha_n (a): n\in \mathbb{N}\}$ coincides with the sequence $\{\mu_n(a): n\in \mathbb{N}\}$ of singular values of $a$, and $\alpha_{-n} (a) = -\mu_n(a)$ for all $n\in \mathbb{N}$.\smallskip

We are interested in an infinite-dimensional version of Mirsky's theorem presented in Theorem \ref{t Mirsky}. An inspection to the proof in \cite[Theorem 9.8]{Bhat2007} points out that this theorem was obtained, via Wielandt's device in $M_n(\mathbb{C})$, from a previous version for hermitian operators also due to Mirsky. We already have a Wielandt's device for matrices in $C_p(H)$, the remaining ingredient for our purposes is the following result due to A.S. Markus.

\begin{theorem}\label{t Markus}{\rm(\cite{Markus64} or \cite[$(41.1)$]{Bhat2007})} Let $a,b$ be compact self-adjoint operators in $B(H)$, and suppose that $a-b\in C_p(H)$ for some $1\leq p\leq \infty$. Let $\{ \alpha_n : n\in \pm \mathbb{N}\}$ and $\{ \beta_n : n\in \pm \mathbb{N}\}$ be enumerations of the eigenvalues of $a$ and $b$, respectively. Then $$\sum_{n=1}^\infty |\alpha_n -\beta_n|^p + \sum_{n=1}^\infty |\alpha_{-n} -\beta_{-n}|^p \leq \|a-b\|_p^p.$$ $\hfill\Box$
\end{theorem}

Suppose now that $a,b\in C_p(H)$. If we apply the previous theorem of Markus to the elements $\widetilde{a} = \left(
                                                                                          \begin{array}{cc}
                                                                                            0 & a \\
                                                                                            a^* & 0 \\
                                                                                          \end{array}
                                                                                        \right)$ and $\widetilde{b} = \left(
                                                                                          \begin{array}{cc}
                                                                                            0 & b \\
                                                                                            b^* & 0 \\
                                                                                          \end{array}
                                                                                        \right)$ in $C_p(H^{(2)}),$
the next result follows from the prior discussion on the Wielandt's device for $C_p(H)$.

\begin{theorem}\label{t Markus non selfadjoint} Let $a,b$ be elements in $C_p(H)$ for some $1\leq p\leq \infty$. Let $\{ \mu_n (a) : n\in \pm \mathbb{N}\}$ and $\{ \mu_n (b) : n\in \pm \mathbb{N}\}$ be the singular values of $a$ and $b$, respectively. Then $$\sum_{n=1}^\infty |\mu_n (a) -\mu_n (b)|^p  \leq \|a-b\|_p^p.$$ \end{theorem}

We observe that Theorem \ref{t Markus non selfadjoint} is an infinite-dimensional version of Mirsky's Theorem \ref{t Mirsky}.\smallskip

We shall finish this note with a couple of applications of the previous Theorem \ref{t Markus non selfadjoint}. We begin with an infinite-dimensional generalization of the claim in \eqref{eq minimum value of fa p=1} for $p=1$ and the result established in \cite[$(11)$ in the proof of Proposition 2.14]{FerJorPer2018} for $1<p<\infty$.

\begin{corollary}\label{c min distance to positive multiples of a min partial isometry} Let $a$ be an element in $C_p(H)$ {\rm(}with $1\leq p<\infty${\rm)}, $t\in \mathbb{R}^+$, and let $\mathcal{U}_{min} (H)\subseteq S(C_p(H))$ denote the set of all minimal partial isometries in $B(H)$, where $H$ is an arbitrary complex Hilbert space. Then the minimum value of the mapping $f_{a,t} : \mathcal{U}_{min} (H) \to \mathbb{R}_0^+$, $ f_{a,t} (e) = \|a- t e\|_p^p$ is  $$ |t-\mu_1(a)|^p + \sum_{j=2}^n \mu_j (a)^p = |t-\mu_1(a)|^p + (\|a\|_p^p-\mu_1(a))^p,$$ where $\{\mu_1(a)\geq \ldots \geq \mu_n(a)\geq \ldots\}$ is the sequence of all singular values of $a$. We further know that this minimum value is attained.
\end{corollary}

\begin{proof} Contrary to the case in which $H$ is finite-dimensional, the set $\mathcal{U}_{min} (H)$ is not norm compact. Let $e$ be an element in $\mathcal{U}_{min} (H)$. Since the singular values of $t e$ are $t$ and zero, a straight application of Theorem \ref{t Markus non selfadjoint} implies that $$ |t-\mu_1(a)|^p + (\|a\|_p^p-\mu_1(a))^p= |t-\mu_1(a)|^p + \sum_{j=2}^n \mu_j (a)^p $$ $$= \sum_{n=1}^\infty |\mu_n (a) -\mu_n (t e)|^p  \leq \|a-te\|_p^p = f_{a,t} (e).$$ The proof culminates by observing that if $\displaystyle a= \sum_{n=1}^{\infty} \mu_n (a) \eta_n \otimes \xi_n,$ where $(\xi_n)$ and $(\eta_n)$ are orthonormal systems in $H$, $(\mu_n (a) )\subset \mathbb{R}_0^{+}$ is the sequence of singular values of $a,$ and $\displaystyle \|a\|_p^p =\sum_{n=1}^{\infty} \mu_n (a)^p,$ then $f_{a,t} (\eta_1 \otimes \xi_1) = |t-\mu_1(a)|^p + (\|a\|_p^p-\mu_1(a)^p).$
\end{proof}

Finally, if in the proof of Theorem \ref{t lambda function in finite dimensional case} we consider an arbitrary complex Hilbert space $H,$ and we replace Claim \eqref{eq minimum value of fa p=1} with the conclusion in Corollary \ref{c min distance to positive multiples of a min partial isometry} for $p=1$ the whole argument in the proof of the just quoted theorem remains valid to obtain a direct and new proof of Theorem \ref{t lambda function in infinite dimensional case}.\smallskip\smallskip\smallskip

\textbf{Acknowledgements} Author partially supported by the Spanish Ministry of Economy and Competitiveness and European Regional Development Fund project no. MTM2014-58984-P and Junta de Andaluc\'{\i}a grant FQM375.
\bigskip


\smallskip


\begin{thebibliography}{99}



\bibitem{AronLohm} R.M. Aron, R. H. Lohman, A geometric function determined by extreme points of the unit ball of a normed space, \emph{Pacific J. Math.} \textbf{127} (1987), 209-231.



\bibitem{Bhat2007} R. Bhatia, \emph{Perturbation bounds for matrix eigenvalues}, Society for Industrial and Applied Mathematics,  Philadelphia, 2007.


\bibitem{BhatSemr96} R. Bhatia, P. Semrl, Distance between hermitian operators in Schatten classes, \emph{Proc. Edinb. Math. Soc.} \textbf{39} (1996), 377-380.

\bibitem{BleLeMerd2004} D.P. Blecher, C. Le Merdy, \emph{Operator algebras and their modules--an operator space approach}, London Mathematical Society Monographs. New Series, 30, Oxford University Press, Oxford, 2004.


\bibitem{BrowPed95} L.G. Brown, G.K. Pedersen, On the geometry of the unit ball of a C$^*$-algebra, \emph{J. Reine Angew. Math.} \textbf{469} (1995), 113-147.

\bibitem{BrowPed97} L.G. Brown, G.K. Pedersen, Approximation and convex decomposition by extremals in C$^*$-algebras, \emph{Math. Scand.} \textbf{81} (1997), 69-85.





\bibitem{BurKaMoPeRa} M. Burgos, A. Kaidi, A. Morales Campoy, A.M. Peralta, M. Ramírez, Von Neumann regularity and quadratic conorms in JB$^*$-triples and C$^*$-algebras, \emph{Acta Math. Sinica} \textbf{24} (2008), 185-200.



\bibitem{DunSchw63} N. Dunford, J.T. Schwartz, \emph{Linear operators. Part II: Spectral theory. Self adjoint operators in Hilbert space}, Interscience Publishers John Wiley \& Sons, New York-London, 1963.



\bibitem{EffRuan2000} E.G. Effros, Z.J. Ruan, \emph{Operator spaces}, London Mathematical Society Monographs. New Series, 23, Oxford University Press, New York, 2000.


\bibitem{FerJorPer2018} F.J. Fernández-Polo, E. Jord{\'a}, A.M. Peralta, Tingley's problem for $p$-Schatten von Neumann classes for $2<p<\infty$, preprint 2018. arXiv:1803.00763v1






\bibitem{GohbergKrein} I.C. Gohberg, M.G. Krein, \emph{Introduction to the theory of linear nonselfadjoint operators}. Translations of Mathematical Monographs, Vol. 18 American Mathematical Society, Providence, R.I., 1969.







\bibitem{TahSiddJam2013} F.B. Jamjoom, A.A. Siddiqui, H.M. Tahlawi,  On the geometry of the unit ball of a JB$^*$-triple, \emph{Abstract and Applied Analysis}, vol. 2013, Article ID 891249, 8 pages, 2013. doi:10.1155/2013/891249

\bibitem{TahSiddJam2014} F.B. Jamjoom, A.A. Siddiqui, H.M. Tahlawi, The $\lambda$-function in JB$^{*}$-triples, \emph{J. Math. Anal. Appl.} \textbf{414} (2014), 734-741.

\bibitem{JamPerSidTah2015} F.B. Jamjoom, A.A. Siddiqui, H.M. Tahlawi, A.M. Peralta, Approximation and convex decomposition by extremals and the $\lambda$-function in JBW$^*$-triples, \emph{Q. J. Math.} (Oxford) \textbf{66} (2015), 583-603.

\bibitem{JamPerSidTah2016} F.B. Jamjoom, A.A. Siddiqui, H.M. Tahlawi, A.M. Peralta, Extremally rich JB$^*$-triples, \emph{Ann. Funct. Anal.} \textbf{7} (2016), 578-592.




\bibitem{Ka83} Kaup, W., A Riemann Mapping Theorem for bounded symmentric domains in complex Banach spaces, \emph{Math. Z.} \textbf{183} (1983), 503-529.







\bibitem{Markus64} A. S. Markus, The eigen and singular values of the sum and product of linear operators, \emph{Russian Math. Surveys} \textbf{19} (1964), 92-120.


\bibitem{McCarthy67} C.A. McCarthy, $C_p$, \emph{Israel J. Math.} \textbf{5} (1967), 249-271.

\bibitem{Mirsky60} L. Mirsky, Symmetric gauge functions and unitarily invariant norms, \emph{Quart. J.
Math. Oxford Ser.} (2), \textbf{11} (1960), 50-59.


\bibitem{Ped91} G.K. Pedersen, The $\lambda$-function in operator algebras, \emph{J. Operator Theory} \textbf{26} (1991), 345-381.



\bibitem{Pi2003} G. Pisier, \emph{Introduction to operator space theory}, London Mathematical Society Lecture Note Series, 294, Cambridge University Press, Cambridge, 2003.



\bibitem{S} S. Sakai, C$^*$-algebras and $W^*$-algebras. Springer Verlag. Berlin--New York, 1971.








\bibitem{Tak} M. Takesaki, \newblock {\em Theory of operator algebras I}, \newblock Springer, Berlin--New York, 2003.




\end{thebibliography}
\end{document}